\documentclass[12pt]{amsart}

\usepackage{a4wide, amsmath,amssymb,amscd,amsthm,txfonts,graphicx,url,bm}
\usepackage{caption}
\usepackage{float}

\usepackage[osf]{Baskervaldx} 
\usepackage{tkz-graph, pgfplots}

\usepackage{color}
\definecolor{orange}{rgb}{.9,.6,.2}
\definecolor{violet}{rgb}{.5,0,.5}
\definecolor{dg}{rgb}{0,0.67,0}

\definecolor{cof}{RGB}{219,144,71}
\definecolor{pur}{RGB}{186,146,162}
\definecolor{greeo}{RGB}{91,173,69}
\definecolor{greet}{RGB}{52,111,72}

\def\QQ{{\mathbb Q}}
\def\ZZ{{\mathbb Z}}

\renewcommand{\bar}{\overline}

\newcommand{\nc}{\newcommand}

\nc{\noi}{\noindent}
\nc{\cmmt}[1]{}

\newcounter{projectcnt}
\nc{\sctn}[1]{{\bigskip\addtocounter{projectcnt}{1}%
\begin{center}{\textbf{Project \theprojectcnt. #1}}\end{center}\medskip}}
\nc{\plainsctn}[1]{{\bigskip\begin{center}{\textbf{#1}}\end{center}\medskip}}

\nc{\dave}[1]{\begin{quote}\em #1\end{quote}}
\nc{\bbP}{{\bf{P}}}
\nc{\N}{{\bf{N}}}
\nc{\bF}{{\bar{\bf{F}}}}
\nc{\Gdave}{{\bf{G}}}
\nc{\I}{{\bf{I}}}
\nc{\cM}{{\cal M}}
\nc{\cK}{{\mathcal K}}
\nc{\cA}{{\cal A}}
\nc{\cO}{{\mathcal O}}
\nc{\cS}{{\mathcal S}}
\nc{\cNF}{{\cal NF}}
\nc{\cMF}{{\cal MF}}
\nc{\cLS}{{\cal LS}}
\nc{\Aff}{\mbox{Aff}}
\nc{\Mod}{{\mbox{\textup{Mod}}}}
\nc{\T}{\mathbb{T}}
\nc{\nf}{S_k^{\textup{\tiny new}}(N, \chi)}
\nc{\of}{S_k^{\textup{\tiny old}}(N, \chi)}
\nc{\mf}{S_k(N, \chi)}
\nc{\Skip}{\mathrm{\rm Skip}}
\nc{\Eis}{\mathrm{Eis}}
\nc{\Poly}{\mathrm{\rm Poly}}
\nc{\Polys}{\mathrm{\rm Polys}}
\nc{\Ext}{\mathrm{\rm Ext}}
\nc{\Ore}{\mathrm{\rm Ore}}
\nc{\Step}{\mathrm{\rm Step}}
\nc{\Mass}{\mathrm{\rm Mass}}
\nc{\mass}{\mathrm{\rm mass}}
\nc{\Vol}{\mathrm{\rm Vol}}

\def\sump{\operatornamewithlimits{\sum\raise7pt\hbox{$\prime$}}}

\DeclareMathOperator{\diag}{diag}

\def\sumpp{\operatornamewithlimits{\sum\raise9pt\hbox{\kern-2pt\scriptsize$(p)$\kern-11pt}\kern6pt}}

\DeclareMathOperator{\tr}{tr}

\newcommand{\Q}{{\mathbb Q}}

\newcommand{\Z}{{\mathbb Z}}

\renewcommand{\P}{{\mathbb P}}

\newtheorem{theorem}{Theorem}[section]
\newtheorem{corollary}[theorem]{Corollary}
\newtheorem{proposition}[theorem]{Proposition}

\newcommand{\litem}{\par\noindent\dimen0=\parindent%
    \advance\dimen0 by-4pt
               \hangindent=\dimen0\ltextindent}

\newcommand{\ltextindent}[1]{\hbox to \hangindent{#1\hss}\ignorespaces}
\newcommand{\ltextjndent}[1]{\hbox to \hangindent{#1\hss}\ignorespaces\kern-1ex}

\newtheorem{remark}[theorem]{Remark}

\begin{document}

\title{Independence Polynomials and Hypergeometric Series}

\author{Danylo~Radchenko}
\address{ETH Zurich, Mathematics Department, Z\"urich 8092, Switzerland}
\email{danradchenko@gmail.com}

\author{Fernando~Rodriguez~Villegas}
\address{The Abdus Salam International Centre for Theoretical Physics\\
Strada Costiera 11\\
Trieste 34151, Italy}
\email{villegas@ictp.it}

\maketitle

\begin{abstract}
  Let $\Gamma$
  be a simple graph and $I_\Gamma(x)$
  its multivariate independence polynomial. The main result of this
  paper is the characterization of chordal graphs as the only $\Gamma$
  for which the power series expansion of $I_\Gamma^{-1}(x)$
  is Horn hypergeometric.
\end{abstract}

\section{}
In this paper by a {\it graph} we will mean a {\it simple graph};
i.e. a usual graph with no multiple edges or loops. Let $\Gamma$ be a
graph on $n$ vertices. We label the vertices of $\Gamma$ and attach to
the $i$-th vertex an independent variable $x_i$. The {\it independence
  polynomial}~\cite[Ch.6]{barnikov} of $\Gamma$ is a polynomial in the
variables $x=(x_1,\ldots, x_n)$ defined as follows.
\begin{equation}
\label{defn-indep-pol}
I_\Gamma(x)=\sum_Ix^I,
\end{equation}
where $I\subseteq \{1,\ldots,n\}$ runs over the independent sets of
vertices of $\Gamma$ and 
$$
x^I:=\prod_{i\in I}  x_i\,.
$$

An {\it independent} set $I\subseteq \{1,\ldots,n\}$ is a subset of
vertices of $\Gamma$ such  that no pair of elements of~$I$ are
connected by  an edge  in $\Gamma$. Note that $I_\Gamma$ has constant
term $1$ for every graph $\Gamma$.

The independence polynomial plays a role in statistical mechanics: it
is the partition function of a lattice gas in the hard-core case; its
vanishing locus is also important because of its connection to the
Lov\'asz local lemma in probability theory (see~\cite{sokal}).

For example, if $\Gamma:=L_n$ is the line graph 
\begin{center}
\begin{tikzpicture}
    \draw (-1,0) node[anchor=east]{$L_n$};

    \draw[fill=black] ( 0, 0) circle (.07);
    \draw[fill=black] ( 1, 0) circle (.07);
    \draw[fill=black] ( 3, 0) circle (.07);
    \draw[fill=black] ( 4, 0) circle (.07);

    \draw[thick] (0,0) -- (1,0);
    \draw[dashed] (1,0) -- (1.5,0);
    \draw[dashed] (2.5,0) -- (3,0);
   \draw[thick] (3,0) -- (4,0);
\end{tikzpicture}
\end{center}
then $I_{\Gamma}(x)=\sum_I x^I$,  where $I\subseteq \{1,\ldots,n\}$
runs over the subsets containing no consecutive numbers $i,i+1$ for
$i=1,\ldots, n-1$. The first few values of $I_{L_n}$ are
\begin{align*}
I_{L_1}&= 1+x_1\,,\\
I_{L_2}&= 1+x_1+x_2\,,\\
I_{L_3}&= 1+x_1+x_2+x_3+x_1x_3\,,\\
I_{L_4}&= 1+x_1+x_2+x_3+x_4+x_1x_3+x_2x_4+x_1x_4\,.
\end{align*}

These polynomials are in fact, up to re-indexing, the multivariate
Fibonacci  polynomials defined by the recursion
\begin{equation}
\label{fibonacci-defn}
F_n=F_{n-1}+x_{n-2}F_{n-2}\,,  \qquad   n >1,\qquad F_0=0, \quad F_1=1.
\end{equation}
We have $I_{L_n}=F_{n+2}$.

A graph $\Gamma$ is called {\it chordal} if it has no induced subgraph
isomorphic to the cycle graph~$C_n$ 
with $n\geq 4$~\cite[Ch.4, \S1]{golumbic}. 
By {\it induced subgraph} defined by a subset $J$ of vertices 
of~$\Gamma$ we mean the subgraph  $\Gamma(J)\subseteq \Gamma$ 
obtained by deleting from~$\Gamma$ the vertices not in $J$ 
and all their attached edges. The {\it cycle graph} $C_n$ consists 
of $n>2$ vertices $1,2,\dots,n$ with an edge joining~$i$ with $i+1$, 
where the indices are read modulo~$n$.

For example, the following  graph is not chordal
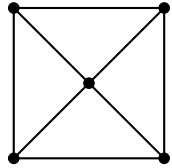
\begin{center}
\begin{tikzpicture}
    \draw[fill=black] ( 0, 0) circle (.07);
    \draw[fill=black] ( 2, 0) circle (.07);
    \draw[fill=black] ( 2, 2) circle (.07);
    \draw[fill=black] ( 0, 2) circle (.07);

    \draw[fill=black] ( 1, 1) circle (.07);

    \draw[thick] (0,0) -- (2,0);
    \draw[thick] (2,0) -- (2,2);
    \draw[thick] (2,2) -- (0,2);
  	\draw[thick] (0,2) -- (0,0);

    \draw[thick] (0,0) -- (1,1);
    \draw[thick] (2,0) -- (1,1);
    \draw[thick] (2,2) -- (1,1);
    \draw[thick] (0,2) -- (1,1);
\end{tikzpicture}
\end{center}
since removing the central vertex  leaves the graph $C_4$
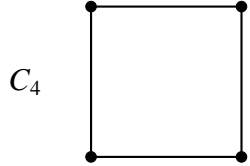
\begin{center}
\begin{tikzpicture}
    \draw (-.5,1) node[anchor=east]{$C_4$};
    \draw[fill=black] ( 0, 0) circle (.07);
    \draw[fill=black] ( 2, 0) circle (.07);
    \draw[fill=black] ( 2, 2) circle (.07);
    \draw[fill=black] ( 0, 2) circle (.07);

    \draw[thick] (0,0) -- (2,0);
   \draw[thick] (2,0) -- (2,2);
   \draw[thick] (2,2) -- (0,2);
   \draw[thick] (0,2) -- (0,0);
\end{tikzpicture}
\end{center}
The  following graph on the other hand is chordal
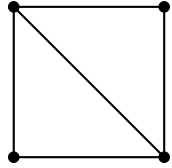
\begin{center}
\begin{tikzpicture}
    \draw[fill=black] ( 0, 0) circle (.07);
    \draw[fill=black] ( 2, 0) circle (.07);
    \draw[fill=black] ( 2, 2) circle (.07);
    \draw[fill=black] ( 0, 2) circle (.07);

    \draw[thick] (0,0) -- (2,0);
   \draw[thick] (2,0) -- (2,2);
   \draw[thick] (2,2) -- (0,2);
   \draw[thick] (0,2) -- (0,0);
   \draw[thick] (0,2) -- (2,0);

\end{tikzpicture}
\end{center}

Finally, a  power series
	$$
	F(x)=\sum_{m\geq 0} c_m x^m, \qquad m=(m_1,\ldots,m_n),  \quad
	x^m:=x_1^{m_1}\cdots  x_n^{m_n}
	$$
is called {\it Horn hypergeometric} if $c_m$ is nonzero
for all $m\ge0$ and
	$$
    \frac{c_{m+e_i}}{c_m},  \qquad
    e_i:=(0,\ldots,\stackrel i 1,\ldots,0) 
	$$
is a rational function of $m_1,\ldots, m_n$ for every $i=1,\ldots, n$.
\begin{remark}
  In the definition of Horn hypergeometric the assumption that $c_m$
  is nonzero could be relaxed (see~\cite{abramov-petkovsek} for a
  general discussion) but it simplifies the arguments and is all we
  will need.
\end{remark}

We can now state our main result.
\begin{theorem}
\label{main-thm}
The following are equivalent.

1) The graph $\Gamma$ is chordal.

2) The power series expansion
	$$
	\frac 1{I_\Gamma(x)}=\sum_{m\geq 0} (-1)^{|m|}c_m x^m, \qquad
	|m|:=m_1+\cdots + m_n\,,
	$$
is Horn hypergeometric.

3)  The power series expansion
	$$
	I_\Gamma(x)^{-s}=\sum_{m\geq 0} (-1)^{|m|}c_m(s) x^m, \qquad
	|m|:=m_1+\cdots + m_n\,,
	$$
is Horn hypergeometric for all $s\not\in\ZZ_{\le0}$.

\end{theorem}

The proof of the main theorem is spread over the next several
sections. In Corollary~\ref{main-thm-1} we prove that 1) $\Rightarrow$ 3). 
We then prove that 2) $\Rightarrow$ 1), which takes longer
and is completed in Proposition~\ref{main-thm-2}.  This finishes the
proof as the remaining implication 3)$\Rightarrow$ 2) is trivial.
We include in the last section~\S\ref{misc} some miscellaneous 
results that arose in the process of proving the main result.

We should mention that by a theorem of
Cartier-Foata~\cite{cartier-foata} the coefficients $c_m$
in Theorem~\ref{main-thm} 2) have a combinatorial interpretation and
are in particular positive integers. Indeed, consider the algebra
$A_\Gamma$
generated over $\Q$ by elements $w_1,\ldots,w_n$ with relations
	$$
	w_iw_j=w_jw_i,
	$$
if and only if $i$ and  $j$ are not connected by an edge in
$\Gamma$. Then  Cartier-Foata~\cite{cartier-foata} prove that  
	\begin{equation}
	\label{cartier-foata}
	\frac 1{\sum_I (-1)^{\#I} w^I}=\sum_J w^J,
	\end{equation}
where the sum  on the left runs over subsets $I\subseteq \{1,\ldots,n\}$
such that all $w_i$  with $i\in I$ commute with each other, whereas
the sum on  the right runs over distinct monomials $w^J$ in the algebra.

Now consider the abelianization map
$$
\begin{array}{cccc}
\Phi:\quad & A_\Gamma &\rightarrow &\Q[x_1,\ldots,x_n]\\
&  w_i& \mapsto &x_i\,.
\end{array}
$$
Applied to the left hand  side of~\eqref{cartier-foata} we obtain
$I_\Gamma(-x)^{-1}$. Hence we deduce that
	$$
	c_m=\#\{ J \,|\,  \Phi(w^J)=x^m\}\,.
	$$
In other words, $c_m$ counts all the  rearrangements  of the monomial
$w_1^{m_1}\cdots w_n^{m_n}$ that give distinct monomials in $A_\Gamma$.

For example, if $\Gamma=K_n$ is the complete graph on $n$ vertices
then $I_{K_n}(x)=1+x_1+\cdots +x_n$ and
$$
\frac 1 {1-x_1-\cdots-x_n}=\sum_{m\geq 0}
\frac{(m_1+\cdots+m_n)!}{m_1!\cdots m_n!}x^m.
$$
In fact, the right hand  side is Horn hypergeometric and this is a
simple instance of the main theorem  since $K_n$ is clearly chordal.

\section{}

To any integral matrix $A\in \Z^{n\times n}$ we associate the
following Nahm system of equations
\begin{equation}
\label{nahm-syst}
1-z_i=x_i\prod_{j=1}^nz_j^{a_{i,j}},\qquad i =1,\ldots, n.
\end{equation}
We call it a Nahm system, because it specializes (for $A$
 symmetric and positive-definite) to the system considered by Nahm
in his conjecture on the modularity of certain associated
$q$-hypergeometric series when $x_i=1$ (see~\cite[p.~42]{nahm},
\cite[eq.~(25)]{zagier-dilog}).  We think of the system as expressing
the $z's$ as algebraic functions of the $x's$ and we are interested in
the corresponding power series expressions for~$z_i$. Note that $z_i=1$
when $x_i=0$, so these power series have constant term equal to $1$.

It follows from the multivariate Lagrange inversion
(see~\cite{frv-A-pol} for details) that for any $s_1,\ldots, s_n$ 
we have
\begin{equation}
\label{z-power-series}
z_1^{s_1}\cdots z_n^{s_n}= \frac1 D
\sum_{m\geq 0} (-1)^{|m|}\prod_{j=1}^n\binom{s_j+a_j(m)}{m_j} x^m,
\end{equation}
where
\begin{equation}
\label{D-power-series}
D:=\sum_{m\geq 0} (-1)^{|m|}\prod_{j=1}^n\binom{a_j(m)}{m_j} x^m,
\end{equation}
and
	$$
	a_j(m)=\sum_{i=1}^na_{i,j}m_j
	$$
are the linear forms  determined by the columns of~$A$. 
Here we interpret the binomial coefficients as polynomials 
of the top entry
	$$
	\binom  xm:=\frac{x(x-1)\cdots(x-m+1)}{m!}\,,
	$$
for any non-negative integer~$m$. We also have
\begin{equation}
\label{D-det}
D^{-1}=\det\left(I_n +
  \diag\left(\frac{1-z_1}{z_1},\ldots,\frac{1-z_n}{z_n}\right)A\right),
\end{equation}
where~$I_n$ is the identity matrix of size~$n$.

If $A$ is upper triangular with $1$'s along the diagonal then we can
recursively solve for the~$z's$ in terms of the~$x's$. In particular,
$z_i$ is a rational function of $x_1,\ldots,x_n$. It also follows
easily from~\eqref{D-det} that in this case 
	\begin{equation}
	\label{D-prod}
	D=z_1\cdots z_n\,.
	\end{equation}

It appears to be rare for non upper triangular matrices~$A$ 
(more precisely, for matrices that are not permutation-similar 
to an upper triangular matrix) to give rise to rational~$z's$,
but it does happen. A simple but interesting example 
(related to the $5$-term relation for the 
dilogarithm~\cite{frv-non-orientable}) is the following. Take
$A=\left(\begin{smallmatrix}0&1\\1&0\end{smallmatrix}\right)$ 
then one easily checks that
	$$
	D=\frac1{1-x_1x_2}\,,\qquad
	z_1=\frac{1-x_1}{1-x_1x_2}\,,\qquad
	z_2=\frac{1-x_2}{1-x_1x_2}\,.
	$$

We have the following recursion for~$D$.
\begin{proposition}
Let $A$ be upper-triangular with
$1$'s along the diagonal. Let $A^*$ be the $(n-1)\times(n-1)$ matrix
obtained by removing the $n$-th row and column of $A$ and let $D^*$ be
the corresponding value of $D$ as in~\eqref{D-power-series} for
$A^*$. Then
\begin{equation}
\label{D-recursion}
D(x_1,\ldots,x_n)=\frac 1{1+x_n}
D^*\left(\frac{x_1}{(1+x_n)^{a_{1,n}}},
  \ldots,\frac{x_{n-1}}{(1+x_n)^{a_{n-1,n}}}\right)\,. 
\end{equation}
\end{proposition}
\begin{proof}
The claim follows from 
$$
\sum_{m\geq 0}(-1)^m\binom{a+m}mx^m=\frac1{(1+x)^{a+1}}\,.
$$
We leave the details to the reader.
\end{proof}
The following corollary is immediate.
\begin{corollary}
If $A$ is an upper-triangular matrix with $1$'s along 
the diagonal, then $D$ is rational with denominator 
of the form
	$$
	\prod_{j=1}^n(1+x_j)^{k_j}
	$$
for certain non-negative integers $k_j$.
\end{corollary}

We associate to a graph $\Gamma$ with $n$ labeled vertices the 
following upper-triangular matrix $A=(a_{i,j})$ with $1's$ along 
the diagonal. 
\begin{equation}
\label{A-defn}
a_{i,j}=
\begin{cases}
1\,, & i=j\,,\\
1\,, & i\sim j\,, \quad i<j\\
0\,,& {\text otherwise}\,,
\end{cases},
\end{equation}
where $i\sim j$ means that the two vertices~$i$ and~$j$ are connected
by an edge in~$\Gamma$. In other words,~$A$ is basically the top half
of the adjacency matrix of~$\Gamma$.

\section{}
\label{perf-elim-ord}
A (reverse) {\it perfect elimination ordering} of the 
vertices of~$\Gamma$ is a labeling of the vertices such 
that for each $1\leq k\leq n$ the 
subgraph $\Gamma_k\subseteq \Gamma$ induced by the set 
of vertices with labels $1\leq i<k$ connected to the $k$-th 
vertex is a complete graph~\cite[Ch.4, \S2]{golumbic},~\cite{fulkerson-gross}.

For example, the following is a perfect elimination ordering of the
graph $\Gamma$.
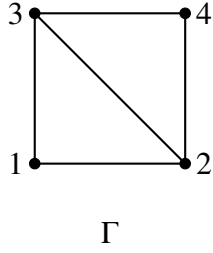
\begin{figure}[H]
\centering
\begin{tikzpicture}

    \draw (0,0) node[anchor=east]{$1$};
    \draw (2,0) node[anchor=west]{$2$};
    \draw (0,2) node[anchor=east]{$3$};
    \draw (2,2) node[anchor=west]{$4$};

    \draw[fill=black] ( 0, 0) circle (.07);
    \draw[fill=black] ( 2, 0) circle (.07);
    \draw[fill=black] ( 0, 2) circle (.07);
    \draw[fill=black] ( 2, 2) circle (.07);

    \draw[thick] (0,0) -- (2,0);
    \draw[thick] (2,0) -- (2,2);
    \draw[thick] (0,0) -- (0,2);
    \draw[thick] (0,2) -- (2,2);
    \draw[thick] (2,0) -- (0,2);
\end{tikzpicture}
\caption{$\Gamma$} 
\end{figure}

These are the corresponding subgraphs $\Gamma_k$.
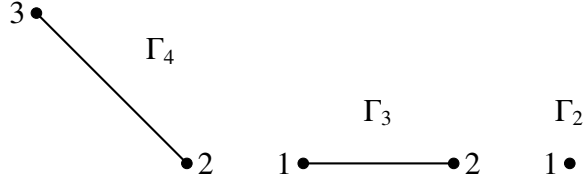
\begin{figure}[H]
\centering
\begin{tikzpicture}
  \draw (2,1.5) node[anchor=east]{$\Gamma_4$};
    \draw (2,0) node[anchor=west]{$2$};
    \draw (0,2) node[anchor=east]{$3$};

    \draw[fill=black] ( 2, 0) circle (.07);
    \draw[fill=black] ( 0, 2) circle (.07);

    \draw[thick] (2,0) -- (0,2);
\end{tikzpicture}
\quad
\begin{tikzpicture}
  \draw (1,1) node[anchor=north]{$\Gamma_3$};
    \draw (0,0) node[anchor=east]{$1$};
    \draw (2,0) node[anchor=west]{$2$};

    \draw[fill=black] ( 0, 0) circle (.07);
    \draw[fill=black] ( 2, 0) circle (.07);

    \draw[thick] (0,0) -- (2,0);

\end{tikzpicture}
\quad
\begin{tikzpicture}
  \draw (0,1) node[anchor=north]{$\Gamma_2$};
    \draw (0,0) node[anchor=east]{$1$};
    \draw[fill=black] ( 0, 0) circle (.07);
\end{tikzpicture}
\end{figure}

\begin{proposition}
Let $\Gamma$ be a graph with a given perfect elimination ordering 
of its vertices. Let $A$ be the corresponding upper triangular 
matrix defined above. Then 
\begin{equation}
\label{D-indep}
D(x_1,\ldots,x_n)=\frac 1{I_\Gamma(x_1,\ldots,x_n)}\,,
\end{equation}
where $I_\Gamma$ is the independence polynomial of~$\Gamma$ 
and~$D$ is defined in~\eqref{D-power-series}.
\end{proposition}
\begin{proof}
We prove the claim by induction in $n$ with the 
recursion~\eqref{D-recursion} as the key step, the case of 
one vertex being trivial. We identify the vertices of~$\Gamma$ 
with $\{1,\ldots,n\}$ using the given perfect elimination 
ordering. Let~$\Gamma^*$ be the graph obtained from~$\Gamma$ 
by deleting the vertex~$n$ and all of its attached edges.

Let $I^*\subseteq \{1,\ldots,n-1\}$
be an independent set of $\Gamma^*$.
It can contain at most one vertex connected to $n$
in $\Gamma$ since by definition of perfect elimination 
ordering any two such vertices are connected by an edge. 
We conclude that an independent set~$I$ of~$\Gamma$ 
properly contains~$I^*$ if and only if no vertex in $I^*$ 
is connected to~$n$, in which case $I=I^*\cup\{n\}$. 

In terms of the independence polynomial this can be 
formulated as follows. Let 
	$$
	y_i:=\begin{cases}
	x_i/(1+x_n)\,,& i\sim n\,,\\
	x_i\,, &{\text otherwise}\,.
	\end{cases}
	$$
Then
	$$
	I_{\Gamma}(x_1,\ldots,x_n)=(1+x_n)I_{\Gamma^*}(y_1,\ldots,y_{n-1})\,.
	$$
This is precisely the recursion satisfied by $\frac{1}{D}$ in terms 
of $\frac{1}{D^*}$ by~\eqref{D-recursion} and the claim follows.
\end{proof}
To prove that 1) implies 3) in Theorem~\ref{main-thm} we need the
following (compare with~\cite{RRV}[\S12.4]).
\begin{corollary}
\label{main-corollary}
With the hypothesis of the proposition we have for all $s$ 
    \begin{equation} \label{eq:main-corollary}
    I_\Gamma(x_1,\ldots,x_n)^{-s}=
    \sum_{m\geq 0} (-1)^{|m|}\prod_{j=1}^n\binom{s-1+a_j(m)}{m_j}\, x^m\,.
    \end{equation}
In particular, if $s$ is not an integer $\le 0$, then
$I_\Gamma(x_1,\ldots,x_n)^{-s}$ is Horn hypergeometric.
\end{corollary}
\begin{proof}
The first claim follows from the above proposition by 
using~\eqref{z-power-series} and~\eqref{D-prod}. 
Since $a_j(m)$ is an integer and $a_j(m)\ge m_j$,
the binomial coefficient in~\eqref{eq:main-corollary} can vanish 
only if~$s$ is a non-positive integer.
\end{proof}

As an example of~\eqref{D-indep} we have the following expansion. For
any positive integer $n$
	$$
	\frac 1{F_{n+2}(x_1,\ldots,x_n)}=
	\sum_{m\geq 0} (-1)^{|m|}
	\prod_{j=2}^{n}\binom{m_j+m_{j-1}}{m_j}\,x^m\,,
	$$
where $F_n$ is the Fibonacci polynomial~\eqref{fibonacci-defn}. 
It is clear that the labeling
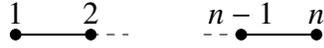
\begin{center}
\begin{tikzpicture}
    \draw (0,0) node[anchor=south]{$1$};
    \draw (1,0) node[anchor=south]{$2$};
    \draw (3,0) node[anchor=south]{$n-1$};
    \draw (4,0) node[anchor=south]{$n$};

    \draw[fill=black] ( 0, 0) circle (.07);
    \draw[fill=black] ( 1, 0) circle (.07);
    \draw[fill=black] ( 3, 0) circle (.07);
    \draw[fill=black] ( 4, 0) circle (.07);

    \draw[thick] (0,0) -- (1,0);
    \draw[dashed] (1,0) -- (1.5,0);
    \draw[dashed] (2.5,0) -- (3,0);
    \draw[thick] (3,0) -- (4,0);

\end{tikzpicture}
\end{center}
of the vertices of $L_n$ is a perfect elimination ordering.

Note that Corollary~\ref{main-corollary} implies, in particular, 
that if a graph $\Gamma$ has a perfect elimination ordering then 
its independence polynomial $I_\Gamma$ satisfies that the expansion 
of $I_\Gamma^{-s}$ in power series is Horn hypergeometric
for all $s\not\in\ZZ_{\le 0}$. Not every
graph has a perfect elimination ordering. It is a 
remarkable fact~\cite[Thm.~4.1]{golumbic} that a graph has a 
perfect elimination ordering if and only if it is chordal.
We conclude the following.
\begin{corollary}
\label{main-thm-1}
Let $\Gamma$  be a chordal graph. Then its independence
polynomial $I_\Gamma$ satisfies that the expansion of 
$I_\Gamma^{-s}$ in a power series is Horn hypergeometric
for all $s\not\in\ZZ_{\le 0}$.
\end{corollary}
This is one direction in our main theorem. To prove the reverse
direction will take a bit more work.

The first observation is that if $\Gamma(J)\subseteq \Gamma$ is the
subgraph  induced by a subset $J$ of its vertices then
$I_{\Gamma(J)}$ is obtained from $I_\Gamma$ by setting $x_j=0$ for
every $j$ not in $J$. It follows that the independence polynomial
$I_\Gamma$ of a  non-chordal graph~$\Gamma$ specializes to the
independence polynomial $I_n$ of the cycle graph $C_n$ for some $n \geq
4$ by setting  appropriate variables equal to zero.

The second observation is that for a power series the property of
being Horn hypergeometric is preserved by the specialization to zero
of any number of its variables. Hence, to finish the proof of the main
theorem it is enough to show that $I_n$ is not Horn
hypergeometric for any $n\geq 4$. 

Notice that $I_3^{-1}(x)$ {\it is} Horn hypergeometric.
Indeed, we have  
	$$
	I_3(x_1,x_2,x_3)=1+x_1+x_2+x_3,\quad \quad
	\frac 1 {I_3(x_1,x_2,x_3)}=\sum_{m_1,m_2,m_3\geq 0}
	(-1)^{|m|}\frac{(m_1+m_2+m_3)!}{m_1!m_2!m_3!} 
	x_1^{m_1}x_2^{m_2}x_3^{m_3}.
	$$ 
In fact, we have a case of the {\it strong law of small numbers}: the
cycle graph $C_n$ and the complete graph $K_n$ coincide for $n=3$ but
not for any $n\geq 4$.

\section{}

\bigskip

Recall that $I_n(x_1,\ldots,x_n)$ denotes the independence 
polynomial of the cycle graph $C_n$ for $n\geq 3$. 
It will be convenient to extend the definition and include 
	$$
	I_1(x_1):=1+x_1\,,\qquad I_2(x_1,x_2):=1+x_1+x_2\,.
	$$

We would like to describe the coefficients in the power series
expansion of $I_n(x)^{-1}$. We will make use of the Nahm
system~\eqref{nahm-syst} associated to the following matrix
	$$
	A:=\left(
	\begin{array}{cccccc}
	1&0&0&\cdots&0&1\\
	1&1&0&\cdots&0&0\\
	0&1&1&\cdots&0&0\\
	\vdots&\vdots&\vdots&\ddots&\vdots&\vdots\\
	0&0&0&\cdots&1&0\\
	0&0&0&\cdots&1&1
	\end{array}
	\right)\,,
	\qquad a_{i,j}:=\begin{cases}
	  1\,,& j  = i \quad {\text or} \quad 
	          j\equiv i-1\bmod n\,,\\
	  0\,,& {\text otherwise}\,.
	\end{cases}
	$$
Namely, consider the system
\begin{equation}
\label{cyclic-syst}
\left\{
\begin{array}{ccc}
1-z_1&=& x_1z_1z_n\,,\\
1-z_2&=&x_2z_2z_1\,,\\
\vdots&\vdots&\vdots\\
1-z_n&=& x_{n-1}z_nz_{n-1}\,.
\end{array}
\right.
\end{equation}
Then
$$
D=
\sum_{m\geq 0}
(-1)^{|m|}\binom{m_1+m_2}{m_1}\binom{m_2+m_3}{m_2}\cdots
\binom{m_{n}+m_{1}}{m_n}\,x^m\,. 
$$
\begin{proposition}
\label{cyclic-syst-fmlae}
Let 
$$
u:=z_1\cdots z_n,\qquad \qquad v:=(-1)^nx_1\cdots x_n
$$
and
\begin{equation}
\label{M-defn}
M:=
\left(
\begin{array}{cc}
1&-1\\-x_1&0
\end{array}
\right)
\left(
\begin{array}{cc}
1&-1\\-x_2&0
\end{array}
\right)
\cdots
\left(
\begin{array}{cc}
1&-1\\-x_n&0
\end{array}
\right)\,.
\end{equation}
Then the following statements hold.
\begin{enumerate}
\item[(i)]
$$
\tr(M)=I_n(x_1,\ldots,x_n), \qquad \det(M)=v.
$$
\item[(ii)]
 $$
M\left(\begin{array}{c}z_n\\1\end{array}\right)=uv
\left(\begin{array}{c}z_n\\1\end{array}\right).
$$
\item[(iii)]
The polynomial $X^2-I_n(x_1,\ldots,x_n)X+v$ has 
roots $u^{-1}$ and~$uv$.
\item[(iv)]
	$$
	I_n(x_1,\dots,x_n)=u^{-1}+uv\,
	$$
\item[(v)]
	$$
	D^{-1}=u^{-1}-uv\,
	$$
\item[(vi)]
	$$
	D^{-2}=I_n(x_1,\ldots,x_n)^2-(-1)^n4x_1\cdots x_n\,.
	$$
\item[(vii)]
	$$
	\frac1{\sqrt{I_n(x_1,\ldots,x_n)^2-(-1)^n4x_1\cdots x_n}}=
	\sum_{m\geq 0}
	(-1)^{|m|}\binom{m_1+m_2}{m_1}\binom{m_2+m_3}{m_2}\cdots
	\binom{m_{n}+m_1}{m_n} x^m. 
	$$
\end{enumerate}
\end{proposition}
\begin{proof}
(i) The second identity is immediate.
For the first identity we expand the trace as
	$$
	\tr(M)=\sum_{i_1,\dots,i_n\in\{1,2\}}
	a_{i_1i_2}^{(1)}\dots a_{i_ni_1}^{(n)}
	\,,
	$$
where $a_{ij}^{(k)}$ are the entries of the $k$-th
matrix in the product defining~$M$, and note that if we 
encode $(i_1,\dots,i_n)$ by $I=\{j\in\{1,\dots,n\}\colon i_j=2\}$, 
then the $I$-th term vanishes if $I$ contains two indices 
consecutive modulo $n$ (since $a_{22}^{(k)}=0$) and is equal 
to $\prod_{i\in I}x_i$ otherwise.

(ii) Note that from the Nahm system~\eqref{cyclic-syst}
we have
$$
\left\{
\begin{array}{ccl}
z_{n-1}&=&(1-z_n)/x_nz_n\\
z_{n-2}&=& (1-z_{n-1})/x_{n-1}z_{n-1}\\
\vdots&\vdots&\quad\vdots\\
z_1&=& (1-z_2)/x_2z_2\\
z_n&=& (1-z_1)/x_1z_1\\
\end{array}
\right.
$$
This is equivalent to the claim.

(iii) It follows from (ii) that $uv$ is an eigenvalue of $M$. 
From the determinant value in (i) the other eigenvalue 
is $u^{-1}$. The quadratic polynomial is the characteristic 
polynomial of $M$ by (i).

(iv) Follows immediately from (iii).

(v) By~\eqref{D-det} using the system equations $D^{-1}$ 
is the determinant of the $n\times n$ matrix $W=(w_{i,j})$ 
with $w_{i,j}=z_i^{-1}$ for $j=i$ and $x_iz_i$ for $j\equiv i-1\bmod n$.
Consequently,
	$$
	D^{-1}=\prod_{i=1}^nz_i^{-1}-(-1)^n\prod_{i=1}^nx_iz_i,
	$$
which is what we wanted to prove.

(vi) From (iii) and (v) we see that $D^{-2}$ is the discriminant of the
quadratic polynomial in (iii) and the claim follows.

(vii) This is just a restatement of (vi).
\end{proof}

The expansion (vii) was proved earlier by
Carlitz~\cite{carlitz},~\cite[\S4.4]{riordan}. 

\section{}
\label{diagonal}
As mentioned, we are interested in the coefficients of the power series
expansion of $I_n(x_1,\ldots,x_n)^{-1}$. To obtain these we will
extend the results of the previous section. Let
	$$
	c_{m,j}:=\binom{m_1+m_2}{m_1+j}\binom{m_2+m_3}{m_2+j}\cdots
	\binom{m_{n}+m_{1}}{m_n+j}\,, \qquad |j|\leq \min(m)\,.
	$$
Note that
	$$
	c_{m,j}=\frac{(m_1+m_2)!\cdots (m_{n-1}+m_n)!(m_n+m_1)!}
	{(m_1+j)!(m_1-j)!\cdots (m_n+j)!(m_n-j)!}.
	$$
In particular, $c_{m,-j}=c_{m,j}$. Let
\begin{equation}
\label{R-power-series}
R(z;x_1,\ldots,x_n):=\sum_{m\geq 0}\sum_{|j|\leq
  \min(m)} (-1)^{|m|}c_{m,j} z^jx^m
\end{equation}
be the generating series of these coefficients.
Note that the coefficient of $z^0$ of $R(z;x_1,\ldots,x_n)$ 
is equal to~$D$.

Fix some non-negative integer $j$. The coefficient of $z^j$ in~$R$ 
can be expressed in the form
    $$
    v^j\sum_{m\geq 0}(-1)^{|m|} \binom{m_1+m_2+2j}{m_1}\binom{m_2+m_3+2j}{m_2}\cdots
    \binom{m_{n}+m_{1}+2j}{m_n} x^m
    $$
after replacing $m_i$ by $m_i-j$, where recall that $v=(-1)^nx_1\cdots
x_n$.  By Lagrange inversion~\eqref{z-power-series} we find that this in turn
equals $Dv^j(z_1\cdots z_n)^{2j}.$

To simplify the notation let $w:=v(z_1\cdots z_n)^2$.  From the
coefficients of $R$ in powers of $z$ we can reconstruct the series;
summing the geometric series we find that
$$
R=D\left(1+\frac{wz}{1-wz}+\frac{wz^{-1}}{1-wz^{-1}}\right).
$$
Alternatively,
\begin{equation}
\label{R-fmla1}
R^{-1}=\frac1D
\left(\frac{1+w}{1-w}-\left(z^{\tfrac12}+z^{-\tfrac12}\right)^2\frac 
  w{1-w^2}\right)=\frac1D \frac{(w-z)(w-z^{-1})}{1-w^2}\,.
\end{equation}

\begin{proposition}
The power series $R(z;x_1,\ldots,x_n)$ is the Taylor expansion of a
rational function. More precisely,
\begin{equation}
\label{R-fmla}
R(z;x_1,\ldots,x_n)=\frac{I_n(x_1,\ldots,x_n)}{I_n(x_1,\ldots,x_n)^2
  -(-1)^n\left(z^{\tfrac12}+z^{-\tfrac12}\right)^2
x_1\cdots x_n}.
\end{equation}
\end{proposition}
\begin{proof}
Using~\eqref{R-fmla1} it is enough to show that
	$$
	I_n(x_1,\ldots,x_n)=\frac1D\left(\frac{1+w}{1-w}\right)=Dv
	\left(\frac{1-w^2}w\right)  
	$$
and this follows easily from 
Proposition~\ref{cyclic-syst-fmlae} noting that $w=vu^2$.
\end{proof}
\begin{corollary}
\label{I_n-power-series}
The following power series expansion holds
\begin{equation}
I_n(x_1,\ldots,x_n)^{-1}=\sum_{m\geq 0}(-1)^{|m|} \sum_{|j|\leq
  \min(m)}(-1)^j \binom{m_1+m_2}{m_1+j}\binom{m_2+m_3}{m_2+j}\cdots
\binom{m_{n}+m_1}{m_n+j}x^m.
\end{equation}
\end{corollary}
\begin{proof}
It follows from the proposition by taking $z=-1$.
\end{proof}

We  are now ready to finish the proof of our main result. 
\begin{proposition}
\label{main-thm-2}
For $n\geq 4$ the power series expansion of $I_n(x)^{-1}$  is not Horn
hypergeometric.  
\end{proposition}
\begin{proof}
Let $c_m$ be the coefficients in the expansion of $I_n(x)^{-1}$
	$$
	I_n(x)^{-1}=\sum_{m\geq 0}(-1)^{|m|} c_m x^m.
	$$
To prove the claim it is enough to show  that if the one variable series
(the main diagonal)
	$$
	H_n(x):=\sum_{k\geq 0}c_kx^k, \qquad \qquad c_k:=c_{(k,\ldots,k)}
	$$
is Horn hypergeometric then $n\leq 3$.

The case $n=1$ being trivial we may  assume $n>1$. By
Corollary~\ref{I_n-power-series} we have 
	$$
	c_k=\sum_{|j|\leq k} (-1)^j\binom{2k}{k+j}^n, \qquad  n >1.
	$$
These numbers are known as de Bruijn numbers in the literature and are
denoted by $S(n,k)$. De Bruijn in his book~\cite{de-bruijn} computed
the asymptotic behaviour of $S(n,k)$ for fixed $n$ and large~$k$. It
follows from his computation that
	$$
	c_{k+1}/c_k\rightarrow  \kappa_n, \qquad k \rightarrow\infty,
	$$
where
	$$
	\kappa_n:=(2\cos(\pi/2n))^{2n}\,.
	$$
We now apply de Bruijn's argument: if $H_k(x)$ is Horn
hypergeometric then $\kappa_n$ has to be rational, and 
hence $n\le 3$. (Here is a short proof of this: 
$2\cos(\pi/2n)$ is an algebraic integer, it generates a real
cyclotomic extension of~$\QQ$ of degree~$\varphi(4n)/2$, and all of 
its $\varphi(4n)/2$ conjugates are real numbers in $(-2,2)$. 
Therefore, if $\kappa_n$ is rational, then there can be at most 
two conjugates, since their absolute values have to be equal, 
and hence $\varphi(4n)\le 4$, thus $n\le 3$.)
\end{proof}

\section{}
\label{misc}
In this section we sketch very briefly several miscellaneous results
stemming from the previous discussion; these will be expanded on in a
later publication.

\medskip
1) We can expand the right hand side of~\eqref{R-fmla} in the variable 
$t:=\frac12\left(z^{\tfrac12}+z^{-\tfrac12}\right)$  and  compare
coefficients   to the left  hand side to obtain some interesting
identities. We will make this  explicit for $n=3$ where the identity
generalizes that of Dixon (corresponding to the appropriate
formulation for  $k=0$).

\begin{proposition}
\label{dixon-gen}
For $k>0$ and $m=(m_1,m_2,m_3)$ a triple of non-negative  integers we
have 
\begin{equation*}
\frac 1{k!}\sum_{j=0}^{\min(m)}(-1)^j \frac{(2j+k)(j+k-1)!}{j!}
\binom{m_1+m_2+k}{m_2+k+j}\binom{m_2+m_3+k}{m_3+k+j}\binom{m_3+m_1+k}{m_1+k+j}
=\frac{(k+m_1+m_2+m_3)!}{k!m_1!m_2!m_3!}\,.
\end{equation*}
\end{proposition}
\begin{proof}
We  give  a  sketch of the proof and leave the details to the
reader. With the definition of~$t$ 
	$$
	T_{2j}(t)=\tfrac12\left(z^j+z^{-j}\right),
	$$
where $T_{2j}(t)$  is the $(2j)$-th Chebyshev polynomial. We have
	$$
	T_{2j}(t)=j\sum_{l=0}^j(-1)^{l+j}\frac{(j+l-1)!}{(j-l)!(2l)!}(2t)^{2l}\,,
	\qquad j>0\,.
	$$
Expanding the right hand side of~\eqref{R-fmla} in the 
variable~$t$ we find in general for any  $l>0$
	$$
	\frac{v^l}{I_n(x)^{2l+1}}\\
	=\frac1{2(2l)!}\sum_{m\geq 0}(-1)^{|m|}
	\sum_{j=l}^{\min(m)}(-1)^{l+j}\frac{(j+l-1)!}{(j-l)!}  
	\binom{m_1+m_2}{m_1+j}\binom{m_2+m_3}{m_2+j}\cdots 
	\binom{m_{n}+m_1}{m_n+j}\,x^m.
	$$
Specializing to $n=3$, expanding both sides in and comparing
coefficients yields the claim for~$k$ even. 
A similar argument works for~$k$ odd.
Alternatively, since both sides of the identity are polynomial
functions in~$k$, we obtain the case of odd $k$ by interpolation.
\end{proof}
It is curious that the visible $4$-fold symmetry  on the right hand
side is far from  clear on the left hand side. 

\bigskip
2) The Horn-Kapranov parametrization determined by the Horn hypergeometric
series in~\eqref{R-power-series} is the following
\begin{equation}
\label{horn-param}
\left\{
\begin{aligned}
\phi_0&= \frac{(\lambda_1+\lambda_0)\cdots(\lambda_n+\lambda_0)}
{(\lambda_1-\lambda_0)\cdots(\lambda_n-\lambda_0)}\\
\phi_1&= -\frac{(\lambda_1-\lambda_0)(\lambda_1+\lambda_0)}
{(\lambda_1+\lambda_2)(\lambda_1+\lambda_n)}\\
\vdots&\qquad\vdots\\
\phi_n&= -\frac{(\lambda_n-\lambda_0)(\lambda_n+\lambda_0)}
{(\lambda_n+\lambda_1)(\lambda_n+\lambda_{n-1})}
\end{aligned}
\right.
\end{equation}
Since the singularities of the series occur at the points of vanishing 
of the denominator of~$R$
\begin{equation}
\label{Delta-defn}
\Delta(z;x_1,\ldots,x_n):= I_n(x_1,\ldots,x_n)^2
-(-1)^n\left(z^{\tfrac12}+z^{-\tfrac12}\right)^2 x_1\cdots x_n 
\end{equation}
we have $\Delta(\phi_0;\phi_1,\ldots,\phi_n)=0$. It follows that
$$
I_n(\phi_1,\ldots,\phi_n)=
\frac{\prod_{i=1}^n(\lambda_i+\lambda_0) +
\prod_{i=1}^n(\lambda_i-\lambda_0)}
{\prod_{i=1}^n(\lambda_i+\lambda_{i+1})}\,, 
$$
where the indices are read modulo $n$. This identity
follows from part~(iv) of Proposition~\ref{cyclic-syst-fmlae},
since for $x_i=\phi_i$ we can solve the Nahm system explicitly by 
taking~$z_i=\frac{\lambda_i+\lambda_{i+1}}{\lambda_{i+1}+\lambda_{0}}$.

If we set  $\lambda_0=0$ and $u_i:=\lambda_{i+1}/\lambda_i$ then
\begin{equation}
\label{u-param}
\left\{
\begin{aligned}
\phi_0&= 1\\
\phi_1&= -1/(1+u_1)(1+u_n^{-1})\\
\vdots&\qquad\vdots\\
\phi_n&= -1/(1+u_n)(1+u_{n-1}^{-1})\\
\end{aligned}
\right.
\end{equation}
If we relax the condition  that $u_1\cdots  u_n=1$ that is  a
consequence of their definition and treat them as independent variables
then plugging in the rational map~\eqref{u-param} we obtain
\begin{equation}
\label{cyclic-u-fmla}
I_n\left(-\frac1{(1+u_1)(1+u_n^{-1})},\ldots,
  -\frac1{(1+u_n)(1+u_{n-1}^{-1})}\right)
= \frac{1+u_1\ldots u_n}
{\prod_{i=1}^n(1+u_i)}\,.
\end{equation}
Again, this identity follows from part~(iv) 
of Proposition~\ref{cyclic-syst-fmlae} by taking~$z_i=1+u_i^{-1}$.
Writing this  identity explicitly for $n=2$ and $n=3$ we find
    $$
    1-\frac1{(1+u_1)(1+u_2^{-1})}-\frac1{(1+u_2)(1+u_1^{-1})}=\frac{1+u_1u_2}
         {(1+u_1)(1+u_2)}
    $$
and
    $$
    1-\frac1{(1+u_1)(1+u_3^{-1})}-\frac1{(1+u_2)(1+u_1^{-1})} -
        \frac1{(1+u_3)(1+u_2^{-1})} =\frac{1+u_1u_2u_3}
         {(1+u_1)(1+u_2)(1+u_3)}\,.
    $$
The naive analogue of this identity does not hold for more than three
variables as the  left hand side no longer is the specialization of~$I_n$.

\bigskip 
3) The varieties defined by the vanishing of
$\Delta(z;x_1,\ldots,x_n)$ seem to be quite interesting. Here we
discuss a few cases in the special case of $z=1$ for small $n$, where
the varieties  are classical. Let  
    $$
    \Delta(x_1,\ldots,x_n):=\Delta(1;x_1,\ldots,x_n)=
    I_n(x_1,\ldots,x_n)^2-(-1)^n4 x_1\cdots x_n\,.
    $$
It is a  polynomial of degree $n$. Let $\Delta_n(x_0,x_1,\ldots,x_n)$ be
the homogenization of  $\Delta$ and $X_n\subseteq\P^n$  its zero
locus.  

For $n=2$ we have
    $$
    \Delta_2(x_0,x_1,x_2)=x_0^2+x_1^2+x_2^2 + 2x_0x_1 +2x_0x_2  -2x_1x_2
    $$
and $X_2\subseteq \P^2$ is a smooth conic.

For $n=3$ we have
    $$
    \Delta_3(x_0,x_1,x_2,x_3)=x_0(x_0+x_1+x_2+x_3)^2+4x_1x_2x_3\,.
    $$
We  find that $X_3\subseteq \P^3$ is a cubic surface with the
four double points
$$
(-1 : 1 : 1 : 1),\quad (-1 : 0 : 0 : 1),\quad (-1 : 0 : 1 : 0),\quad
(-1 : 1 : 0 : 0).
$$
It follows that $X_3$ is projectively isomorphic to the Cayley
surface~\cite[p.~500]{dolgachev},~\cite[p.~75]{hunt}.

For $n=4$ we find that $X_4\subseteq \P^4$ is a quartic threefold
non-singular except for $15$ lines. These lines meet in appropriate
groups of three lines at $15$ points. The resulting configuration is
known as the Cremona-Richmond configuration~\cite[\S 9]{coxeter}.  The
variety $X_4$ is isomorphic to the Castelnuovo-Richmond
quartic~\cite[p.~532]{dolgachev} (also known as the Igusa
quartic~\cite[\S 3.3]{hunt}).

\bigskip
4) Let us analyze the Nahm system~\eqref{cyclic-syst} 
a bit more closely. We have the following.
\begin{proposition}
\label{quadr-ext}
\begin{enumerate}
\item[(i)]
Let $n>1$. The Nahm system~\eqref{cyclic-syst} has the following
solution in~$K:=F\left(\sqrt{\Delta}\right)$, where
$F:=\Q(x_1,\ldots,x_n)$ and   
	$$
	\Delta(x_1,\ldots,x_n):= I_n(x_1,\ldots,x_n)^2-(-1)^n4 x_1\cdots x_n.
	$$
For $i=1,\ldots,n$
\begin{equation}
\label{z-fmla}
z_i=\frac{-b_i+\sqrt{\Delta}}{2a_i},\qquad
a_i:=x_{i}F_n(x_{i+2},x_{i+3},\ldots,x_{i-1}), \quad
b_i:=I_n(x_1,\ldots,-x_{i},\ldots,x_n),
\end{equation}
where $F_n$ is the Fibonacci polynomial~\eqref{fibonacci-defn} in
$n-2$  variables.
\item[(ii)]
For $i=1,\ldots, n$ let  $c_i:=-F_n(x_{i},x_{i+1},\ldots,x_{i-3})$ then
	$$
	\Delta=b_i^2-4a_ic_i.
	$$
\end{enumerate}
\end{proposition}

We see that $z_1,\ldots,z_n$ are rational functions on the double
cover~$Z_n$ of $\P^n$  ramified at~$X_n$. These double covers
are also classical varieties: $Z_3$ is Segre's
primal~\cite[p. 530]{dolgachev},~\cite[\S 3.2]{hunt} and $Z_4$
is Coble's variety~\cite[\S 3.5]{hunt}.

\bigskip
 5) There is a connection between the varieties $X_n$
of 4) and wild character varieties~\cite[\S 5]{boalch}. Consider the
following matrix in invertible variables $y_1,\ldots,y_n$
\begin{equation}
\label{Y-defn}
Y:=
\left(
\begin{array}{cc}
y_1^{-1}&1\\1&0
\end{array}
\right)
\left(
\begin{array}{cc}
y_2^{-1}&1\\1&0
\end{array}
\right)
\cdots
\left(
\begin{array}{cc}
y_n^{-1}&1\\1&0
\end{array}
\right)\,.
\end{equation}
If we insert in between each pair of factors the diagonal matrices
$$
\left(\begin{array}{cc}
-y_i&0\\0&1
\end{array}
\right)
\left(\begin{array}{cc}
1&0\\0&-y_i
\end{array}
\right),
\qquad i=1,\ldots,n
$$
using that 
$$
\left(\begin{array}{cc}
1&0\\0&-x
\end{array}
\right)
\left(\begin{array}{cc}
y^{-1}&1\\1&0
\end{array}
\right)
\left(\begin{array}{cc}
-y&0\\0&1
\end{array}
\right)=
\left(\begin{array}{cc}
-1&1\\xy&0
\end{array}
\right)
$$
we obtain
\begin{equation}
\label{Y-M}
\left(\begin{array}{cc}
1&0\\0&-y_n
\end{array}
\right)
Y
\left(\begin{array}{cc}
1&0\\0&-y_n
\end{array}
\right)^{-1}=
\left(
\begin{array}{cc}
y^{-1}&0\\0&y^{-1}
\end{array}
\right)
\left(
\begin{array}{cc}
-1&1\\y_1y_n&0
\end{array}
\right)
\left(
\begin{array}{cc}
-1&1\\y_2y_1&0
\end{array}
\right)
\cdots
\left(
\begin{array}{cc}
-1&1\\y_ny_{n-1}&0
\end{array}
\right)\,,
\end{equation}
where $y:=(-1)^ny_1\cdots y_n$. 

Assume now that $n=2k$ is even. We may insert
$P=\left(\begin{array}{cc}
          0&1\\1&0\end{array}\right)$ 
appropriately in the definition of $Y$~\eqref{Y-defn} and find that 
\begin{equation}
\label{Y-wild}
Y:=
\left(
\begin{array}{cc}
1&y_1^{-1}\\0&1
\end{array}
\right)
\left(
\begin{array}{cc}
1&0\\y_2^{-1}&1
\end{array}
\right)
\cdots
\left(
\begin{array}{cc}
1&y_{n-1}^{-1}\\0&1
\end{array}
\right)
\left(
\begin{array}{cc}
1&0\\y_n^{-1}&1
\end{array}
\right)\,.
\end{equation}
It follows by~\eqref{Y-M} that the matrices $Y$
and $M$
of~\eqref{M-defn} are related and hence by~\eqref{Y-wild} $M$
is related to the equations involved in the definition of certain wild
character varieties (loc.cit.). The difference is that we impose the
condition that the characteristic polynomial of $M$
has a double root instead of prescribing its entries. This should
correspond to taking the Zariski closure of the $2\times 2$
Jordan block instead of a torus element as the target of the moment
map.

\end{document}